\documentclass[11pt, a4paper]{amsart}
\usepackage{a4,amsmath,amssymb,epic}
\usepackage{graphicx}
\usepackage{multicol}
\usepackage{rotating}
\usepackage{lscape}
\usepackage{amssymb,amsthm}
\usepackage{enumitem}
\usepackage[all]{xy}
\usepackage{tikz}
\usepackage{young}
\usepackage[vcentermath,enableskew]{youngtab}
\allowdisplaybreaks
\numberwithin{equation}{subsection}
\theoremstyle{plain}
\newtheorem{lem}{Lemma}[section]
\newtheorem{prop}[lem]{Proposition}

\newtheorem{thm}[lem]{Theorem}
\newtheorem{cor}[lem]{Corollary}

\newtheorem*{lem*}{Lemma}
\newtheorem*{Acknowledgements*}{Acknowledgements}
\newtheorem*{prop*}{Proposition}
\newtheorem*{thm*}{Theorem}

\newtheorem*{cor*}{Corollary}
\newtheorem*{conj*}{Conjecture}
\theoremstyle{remark}

\newtheorem{rmk}[lem]{Remark}
\newtheorem{defn}[lem]{Definition}

\newtheorem{eg}[lem]{Example}

\newtheorem{rem}[lem]{Remark}

\newcommand{\Hom}{\operatorname{Hom}}

\newcommand{\End}{\operatorname{End}}

\newcommand{\ct}{\operatorname{ct}}
\newcommand{\ZZ}{{\mathbb Z}}

\newcommand{\GL}{\mathrm{GL}}

\newcommand{\Sp}{\textbf{\rm \textbf{S}}}

\newcommand{\minusone}{$-1$}

\begin{document}
\title[The partition algebra and the Kronecker coefficients]{The partition algebra and the \\ Kronecker coefficients}
\author{C.~Bowman}
 \address{Institut de Math\'ematiques de Jussieu, 5 rue du Thomas Mann,
  75013, Paris, France} 
  \email{Bowman@math.jussieu.fr}
 \author{M.~De Visscher}
\address{Centre for Mathematical Science,
 City University London,
 Northampton Square,
 London,
 EC1V 0HB,
 England.}
 \email{Maud.Devisscher.1@city.ac.uk }
  \author{R.~Orellana}
\address{Department of Mathematics,
Dartmouth College,
6188 Kemeny Hall, 
Hanover, NH 03755, USA } 
\email{Rosa.C.Orellana@dartmouth.edu}
 
\subjclass[2000]{20C30} 
\date{\today}

 \maketitle
 
 \begin{abstract}
  
  We propose a new approach to study the Kronecker coefficients by using the Schur-Weyl duality between the symmetric group and the partition algebra.  We explain the limiting behavior  and associated bounds in the context of the partition algebra.  Our analysis leads to a uniform description of the Kronecker coefficients when one of the indexing partitions is a hook or a two-part partition. 
  
  
 \end{abstract}

\section*{Introduction}

A fundamental problem in the representation theory of the symmetric group is to describe the coefficients in the decomposition of the tensor product of two Specht modules.   These coefficients are known in the literature as the {\em Kronecker coefficients}.  They are labelled by triples of partitions.  Finding a formula or combinatorial interpretation for these coefficients has been described by Richard Stanley as `one of the main problems in the combinatorial representation theory of the symmetric group'.   This question has received the attention of Littlewood \cite{littlewood}, James \cite[Chapter 2.9]{jk}, Lascoux \cite{Lascoux},  Thibon \cite{Thibon}, Garsia and Remmel \cite{GR},  Kleshchev and Bessenrodt \cite{BK} amongst others and yet  a combinatorial solution has remained beyond reach for over a hundred years.   
 
Murnaghan discovered an amazing limiting phenomenon  satisfied by the Kronecker coefficients; as we increase the length of the first row of the indexing partitions the sequence of Kronecker coefficients obtained stabilises. The limits of these sequences are known as the {\em reduced Kronecker coefficients}. 

 The novel idea of this paper is to study the Kronecker and reduced coefficients through the Schur--Weyl duality between the symmetric group, $\mathfrak{S}_n$, and the partition algebra, $P_r(n)$.  The key observation being that the tensor product of Specht modules corresponds to the restriction of simple modules in $P_r(n)$ to a Young subalgebra.  The combinatorics underlying the representation theory of both objects is based on partitions.   The duality results in a Schur functor,  ${\rm F} : \mathfrak{S}_n{\text{-mod}}   \to  P_r(n){\text{-mod}}, $
which acts by \emph{first row removal} on the partitions labelling the simple modules.   
We exploit this functor along with the following three key facts concerning the representation theory of the partition algebra: (a) it is  semisimple for large $n$ (b) it has a stratification by symmetric groups (c) its non-semisimple representation theory is well developed.

We interpret the Kronecker and reduced Kronecker coefficients and the passage between them in terms of the representation theory of the partition algebra.  The limiting phenomenon discovered by Murnaghan and some associated bounds (due to Brion)  are then naturally explained by the fact that $P_r(n)$ is semisimple for large enough $n$.  

Closed formulas for Kronecker coefficients have only been obtained for triples of partitions with (i) one 2-part partition (ii) two hook partitions and (iii) a hook and a 2-part partition.
  We give a unified simple approach which covers all these cases and generalises further to triples  with one hook partition.


  Our approach brings  forward a general tool to study these coefficients and provides a natural framework for the study of the outstanding problems in the area.  In particular, one should notice that our proofs are surprisingly elementary. 
 
The paper is organised as follows. 
In Sections 1 and 2 we recall the combinatorics underlying the representation theories of the symmetric group and partition algebra.
  In Section 3 we show how to pass the Kronecker problem through Schur--Weyl duality and phrase it as a question concerning the partition algebra.  We then summarise  results concerning the Kronecker and reduced Kronecker coefficients that have a natural interpretation  in this setting. 
   Section 4 contains   a description of the restriction of a standard module for $P_r(n)$ to a Young subalgebra, giving a new representation theoretic interpretation of \cite[Lemma 2.1]{rosa}.   
In   Section 5 we specialise to hook and two-part partitions and obtain closed positive formulas in these cases.
   Section 6 contains an extended example.

%
%
  
\section{Symmetric group combinatorics}

The combinatorics underlying the representation theory of the symmetric group, $\mathfrak{S}_n$, is based on  partitions.  A \emph{partition} $\lambda$ of $n$, denoted $\lambda\vdash n$, is defined to be a weakly decreasing sequence $\lambda=(\lambda_1,\lambda_2,\dots,\lambda_\ell)$ of non-negative integers such that the sum $|\lambda|=\lambda_1+\lambda_2+\dots +\lambda_\ell$ equals $n$.  The {\em length} of a partition is the number of nonzero parts, we denote this by $\ell(\lambda)$.  We let $\Lambda_n$ denote the set of all partitions of $n$.

With a partition, $\lambda$, is associated its \emph{Young diagram}, which is the set of nodes
\[[\lambda]=\left\{(i,j)\in\ZZ_{>0}^2\ \left|\ j\leq \lambda_i\right.\right\}.\]
  Given a node  specified by $i,j\geq1$, we say the node has \emph{content}  $j-i$.  We let $\ct(\lambda_i)$ denote the content of the last node in the $i$th row of $[\lambda]$, that is $\ct(\lambda_i)=\lambda_i-i$.

Over the complex numbers, the irreducible \emph{Specht} modules, $\Sp(\lambda)$, of $\mathfrak{S}_n$ are indexed by the partitions, $\lambda$, of $ n$.  An explicit construction of these modules is given in \cite{jk}.


\subsection{The classical Littlewood--Richardson rule}
The Littlewood--Richardson rule is a combinatorial description of the restriction of a Specht module to a Young subgroup of the symmetric group.  
Through Schur--Weyl duality, the  rule also computes the decomposition of a tensor product of two simple modules of $\GL_n(\mathbb{C})$.  
The Littlewood--Richardson rule is the most famous algorithm for decomposing tensor products and has been generalised in several directions.

The following is a simple restatement of this rule as it appears in \cite[Section 2.8.13]{jk}.

\begin{thm}[The Littlewood--Richardson Rule]
 For $\lambda\vdash r_1$, $\mu\vdash r_2$ and $\nu \vdash r_1+r_2$,
\begin{align*}
\Sp(\nu) \!\!\downarrow_{\mathfrak{S}_{r_1} \times \mathfrak{S}_{r_2}}^{\mathfrak{S}_{r_1+r_2}} \cong \bigoplus_{ \lambda\vdash r_1, \mu\vdash r_2}c^\nu_{\lambda, \mu} \Sp(\lambda)\boxtimes\Sp(\mu)
\end{align*}
where the $c^\nu_{\lambda, \mu}$ are the Littlewood--Richardson coefficients (defined below).
\end{thm}
The Littlewood--Richardson coefficient $c^\nu_{\lambda, \mu}$ is zero, unless $\lambda \subseteq  \nu$, and, otherwise may be calculated as follows.
For each node $(i,j)$ of $\mu$, take a symbol $u_{i,j}$.  Begin with the diagram $\lambda$ and:
\begin{enumerate}[leftmargin=*,itemsep=0.1em]
\item Add to it all symbols $u_{1,j}$ (corresponding to the first row of nodes of $\mu$) in such a way as to produce the diagram of a partition and to satisfy (3).
\item Next add all symbols $u_{2,j}$ (corresponding to the second row of nodes of $\mu$) following the same rules.  Continue this process with all rows of $\mu$.
\item The added symbols must satisfy: $(a)$ for all $i$, if $y <j$, $u_{i,y}$ is in a later column than $u_{i,j}$; and $(b)$ for all $j$, if $x<i$ $u_{x,j}$ is in an earlier row than $u_{i,j}$.
\end{enumerate}

By transitivity of induction we have that the Littlewood--Richardson rule determines the structure of  the restriction of a Specht module to any Young subgroup.  Of particular importance in this paper is the three-part case
\begin{align*}
\Sp(\nu)\!\! \downarrow_{\mathfrak{S}_{r_1} \times \mathfrak{S}_{r_2}\times\mathfrak{S}_{r_3} }^{\mathfrak{S}_{r_1+r_2+r_3}}
 &\cong\bigoplus_{\begin{subarray}{c} {{\xi\vdash r_1+r_2} }
    \\   {{\nu\vdash r_3} } \end{subarray}}
 (  c^\nu_{\xi, \eta}\Sp(\xi)\boxtimes\Sp(\eta))\!\!\downarrow^{\mathfrak{S}_{r_1+r_2} \times \mathfrak{S}_{r_3}}_{\mathfrak{S}_{r_1}\times\mathfrak{S}_{r_2}\times\mathfrak{S}_{r_3}\ }\\
&
\cong\bigoplus_{\begin{subarray}{c} {{\lambda\vdash r_1, } }
    \\   {{\mu\vdash r_2 } }, {{\nu\vdash r_3} }   \end{subarray}}
\left( \sum_{\xi\vdash r_1+r_2}c^\xi_{\lambda,\mu} c^\nu_{\xi,\eta}\right)\Sp(\lambda)\boxtimes\Sp(\mu)\boxtimes\Sp(\eta).
\end{align*}
We therefore set $c^\nu_{\lambda,\mu,\eta}=\sum_{\xi}c^\xi_{\lambda,\mu} c^\nu_{\xi,\eta}$.


\subsection{Tensor products of Specht modules of the symmetric group}\label{murg} In this section we define the Kronecker coefficients and the {reduced Kronecker coefficients} as well as set some notation.  Let $\lambda$ and $\mu$ be two partitions of $n$, then 
\[\Sp(\lambda)\otimes \Sp(\mu) = \bigoplus_{\nu\vdash n} g_{\lambda,\mu}^\nu \Sp(\nu),\]
the coefficients $g_{\lambda , \mu}^\nu$ are known as the {\em Kronecker coefficients}.  These coefficients satisfy an amazing stability property illustrated in the following example. 
\begin{eg}
We have the following tensor products of Specht modules:
\begin{align*}  
\Sp(1^2) \otimes \Sp(1^2) &= \Sp(2) \\
\Sp(2,1) \otimes \Sp(2,1) &= \Sp(3) \oplus \Sp(2,1) \oplus \Sp(1^3) \\
\Sp(3,1) \otimes \Sp(3,1) &= \Sp(4) \oplus \Sp(3,1)\oplus \Sp(2,1^2) \oplus \Sp(2^2)  
\intertext{at which point the product stabilises, i.e. for all $n\geq4$, we have}
\Sp(n-1,1) \otimes \Sp(n-1,1) &= \Sp(n) \oplus \Sp(n-1,1)\oplus \Sp(n-2,1^2) \oplus \Sp(n-2,2).
\end{align*}
  \end{eg}
Let $\lambda = (\lambda_1,\lambda_2, \ldots,\lambda_{\ell} )$ be a partition and $n$ be an integer, define $\lambda_{[n]}=(n-|\lambda|, \lambda_1,\lambda_2, \ldots,\lambda_{\ell})$.  Note that all partitions of $n$ can be written in this form.

For $\lambda_{[n]}, \mu_{[n]}, \nu_{[n]} \in \Lambda_n$ we let 
$$g_{\lambda_{[n]},\mu_{[n]}}^{\nu_{[n]}} = \dim_{\mathbb{C}}(\Hom_{\mathfrak{S}_n}(\Sp(\lambda_{[n]}) \otimes \Sp(\mu_{[n]}) , \Sp(\nu_{[n]}))),$$
denote the multiplicity of $\Sp(\nu_{[n]})$ in the tensor product $\Sp(\lambda_{[n]}) \otimes \Sp(\mu_{[n]})$.
Murnaghan showed (see \cite{murn, murn2}) that if we allow the first parts of the partitions to increase in length then we obtain a limiting behaviour as follows.  For $\lambda_{[N]}, \mu_{[N]}, \nu_{[N]} \in \Lambda_N$ and $N$ sufficiently large we have that
$$g_{\lambda_{[N+k]},\mu_{[N+k]}}^{\nu_{[N+k]}}=\overline{g}_{\lambda,\mu}^\nu$$
for all $k\geq 1$; the integers $\overline{g}_{\lambda,\mu}^\nu$ are  called the \emph{reduced Kronecker coefficients}.
Bounds for this stability have been given in \cite{brion,vallejo,klyachko,rosa}.

\begin{rem}
The reduced Kronecker coefficients are also the structural constants for a linear basis for the polynomials in countably many variables known as the \emph{character polynomials}, see \cite{mac}.
\end{rem}

\section{The partition algebra}
The partition algebra was originally defined by Martin in \cite{marbook}.  All the results in this section are due to Martin and his collaborators, see \cite{mar1} and references therein.  
 \subsection{Definitions}
For $r\in \mathbb{Z}_{>0}$, $\delta\in\mathbb{C}$, we let $P_r(\delta)$ denote the complex vector space with basis given by  all set-partitions of  $\{1,2,\ldots, r, \bar{1},\bar{2}, \ldots, \bar{r}\}.$  
A part of a set-partition is called a \emph{block}.
For example,
$$d=\{\{1, 2, 4, \bar{2}, \bar{5}\}, \{3\}, \{5, 6, 7, \bar{3}, \bar{4}, \bar{6}, \bar{7}\}, \{8, \bar{8}\}, \{\bar{1}\}\},$$
is a set-partition (for $r=8$) with 5  blocks.

A  set-partition can be represented \emph{uniquely} by an $(r,r)$-\emph{partition diagram} consisting of a frame with $r$ distinguished points on the northern and southern boundaries, which we call vertices.  We number the northern vertices from left to right by $1,2,\ldots, r$ and the southern vertices similarly by $\bar{1},\bar{2},\ldots, \bar{r}$. 
Any block in a set-partition is of the form $A\cup B$ where $A=\{i_1<i_2<\ldots <i_p\}$ and $B=\{\bar{j_1}<\bar{j_2}<\ldots <\bar{j_q}\}$ (and $A$ or $B$ could be empty). 
We  draw this block by putting an arc joining each pair $(i_l, i_{l+1})$ and $(\bar{j_l}, \bar{j }_{l+1})$ and if $A$ and $B$ are non-empty we draw a strand from $i_1$ to $\bar{j_1}$, that is we draw a single propagating line on the leftmost vertices of the block.
Blocks containing a northern and a southern vertex will be called \emph{propagating blocks}; all other blocks will be called \emph{non-propagating blocks}.
For $d$ as in the example above, the {partition diagram} of $d$ is given by:
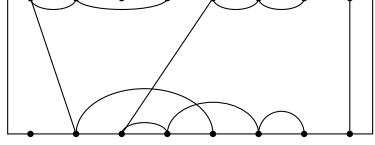
\[ 
\begin{tikzpicture}[scale=0.6]
  \draw (0,0) rectangle (8,3);
  \foreach \x in {0.5,1.5,...,7.5}
    {\fill (\x,3) circle (2pt);
     \fill (\x,0) circle (2pt);}
  \begin{scope}
    \draw (0.5,3) -- (1.5,0);
    \draw (7.5,3) -- (7.5,0);
    \draw (4.5,3) -- (2.5,0);
    \draw (0.5,3) arc (180:360:0.5 and 0.25);
    \draw (1.5,3) arc (180:360:1 and 0.25);
    \draw (4.5,0) arc (0:180:1.5 and 1);
    \draw (5.5,0) arc (0:180:1 and .7);
    \draw (3.5,0) arc (0:180:.5 and .25);
    \draw (6.5,0) arc (0:180:0.5 and 0.5);
    \draw (4.5,3) arc (180:360:0.5 and 0.25);
    \draw (5.5,3) arc (180:360:0.5 and 0.25);
  \end{scope}
\end{tikzpicture}
\]

We can generalise this definition to $(r,m)$-partition diagrams as diagrams representing set-partitions of $\{1, \ldots , r, \bar{1}, \ldots, \bar{m}\}$ in the obvious way.

We define the product $x \cdot y$ of two diagrams $x$
and $y$ using the concatenation of $x$ above $y$, where we identify
the southern vertices of $x$ with the northern vertices of $y$.   
If there are $t$ connected components consisting only of  middle vertices, then the product is set equal to $\delta^t$ times the diagram  
with the middle components removed. Extending this by linearity defines a multiplication on $P_r(\delta)$.

\noindent \textbf{Assumption:} \emph{We assume throughout the paper that   $\delta\neq 0$.  }

\noindent  The following elements of the partition algebra will be of importance.
 \begin{align*}
s_{i,j}=\begin{minipage}{34mm}
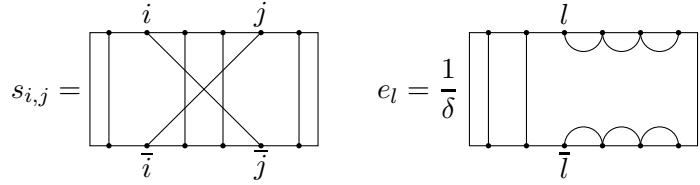
\begin{tikzpicture}[scale=0.5]
  \draw (0,0) rectangle (6,3);
  \foreach \x in {0.5,1.5,...,5.5}
    {\fill (\x,3) circle (2pt);
     \fill (\x,0) circle (2pt);}
    \draw (4.5,3.5) node {$j$};
    \draw (4.5,-0.5) node {$\overline{j}$};
    \draw (1.5,3.5) node {$i$};
    \draw (1.5,-0.5) node {$\overline{i}$};
  \begin{scope}
    \draw (0.5,3) -- (0.5,0);
    \draw (5.5,3) -- (5.5,0);
        \draw (1.5,3) -- (4.5,0);
            \draw (4.5,3) -- (1.5,0);
    \draw (2.5,3) -- (2.5,0);
    \draw (3.5,3) -- (3.5,0);
   \end{scope}
\end{tikzpicture}\end{minipage} \quad
e_{l}= \frac{1}{\delta}\begin{minipage}{34mm} \ \begin{tikzpicture}[scale=0.5]
  \draw (0,0) rectangle (6,3);
  \foreach \x in {0.5,1.5,...,5.5}
    {\fill (\x,3) circle (2pt);
     \fill (\x,0) circle (2pt);}
         \draw (2.5,3.5) node {$l$};
              \draw (2.5,-.5) node {$\overline{l}$};
  \begin{scope}
    \draw (0.5,3) -- (0.5,0);
         \draw (3.5,0) arc (0:180:.5 and .5);
        \draw (4.5,0) arc (0:180:.5 and .5);
        \draw (5.5,0) arc (0:180:.5 and .5);
       \draw (4.5,3) arc (180:360:0.5 and 0.5);
       \draw (3.5,3) arc (180:360:0.5 and 0.5);
              \draw (2.5,3) arc (180:360:0.5 and 0.5);
        \draw (1.5,3) -- (1.5,0);
           \end{scope}
\end{tikzpicture}\end{minipage}
\end{align*}

\noindent In particular, note that   $e_r$ is the idempotent corresponding to the set-partition $\{1,\bar{1}\}\{2,\bar{2}\}\cdots \{r-1,\overline{r-1}\}\{r\}\{\bar{r}\}$. 

\subsection{Filtration by propagating blocks and standard modules}
Fix $\delta\in\mathbb{C}^{\times}$ and write $P_r=P_r(\delta)$. Note that the multiplication in $P_r$ cannot increase the number of propagating blocks.  More precisely, if $x$,  respectively $y$, is a partition diagram with $p_x$, respectively $p_y$, propagating blocks then $x \cdot y$ is equal to $\delta^t z$ for some $t\geq0$ and some partition diagram $z$ with $p_z\leq \min\{p_x, p_y\}$.  This gives a filtration of the algebra $P_r$ by the number of propagating blocks.   This filtration can be realised using the idempotents $e_l$.  We have  
$$\mathbb{C}\cong P_r e_1 P_r \subset \ldots \subset P_re_{r-1}P_r \subset P_re_rP_r \subset P_r.				$$
It is easy to see that
\begin{equation}\label{bob}
e_rP_r e_r \cong P_{r-1} \end{equation}
and that this generalises to $P_{r-l}\cong e_{r-l+1}P_re_{r-l+1}$ for $1\leq l\leq r$. We also have
\begin{equation}\label{bob2} P_r /(P_r e_rP_r )\cong \mathbb{C}\mathfrak{S}_r.\end{equation}
Using equation (\ref{bob2}), we get that any $\mathbb{C}\mathfrak{S}_r$-module can be \emph{inflated} to a $P_r$-module.  We also get from equations (\ref{bob}) and (\ref{bob2}), by induction, that the simple $P_r$-modules are indexed by the set $\Lambda_{\leq r}=\bigcup_{0\leq i\leq r}\Lambda_i$.  

For any $\nu \in \Lambda_{\leq r}$ with $\nu \vdash r-l$, we define a $P_r$-module, $\Delta_r(\nu)$, by
$$\Delta_r(\nu)=P_re_{r-l+1} \otimes_{P_{r-l}}\Sp(\nu).$$
(Here we have identified $P_{r-l}$ with $e_{r-l+1}P_re_{r-l+1}$ using the isomorphism given in equation (\ref{bob}).)  

For $\delta \not\in \{0,1,\ldots, 2r-2\}$ the algebra $P_r(\delta)$ is semisimple and the set $\{\Delta_r(\nu):\nu \in \Lambda_{\leq r}\}$ forms a complete set of non-isomorphic simple modules.  

In general, the algebra $P_r(\delta)$ is quasi-hereditary   with respect to the partial order on $\Lambda_{\leq r}$ given by $\lambda < \mu$ if $|\lambda| > |\mu|$ (see \cite{mar1}).  The modules $\Delta_r(\nu)$ are the standard modules, each of which has a simple head $L_r(\nu)$, and the set  $\{L_r(\nu): \nu\in \Lambda_{\leq r}\}$ forms a complete set of non-isomorphic simple modules.
 
We now give an explicit description of the standard modules. We set $V(r,r-l)$ to be the span of all $(r,r-l)$-partition diagrams having precisely $(r-l)$ propagating blocks. This has a natural structure of a $(P_r(\delta), \mathfrak{S}_{r-l})$-bimodule. It is easy to see that, as vector spaces, we have
$$\Delta_r(\nu) \cong V(r,r-l)\otimes_{\mathfrak{S}_{r-l}} \Sp(\nu).$$ 
The action of $P_r(\delta)$ is given as follows. Let $v$ be a partition diagram in $V(r,r-l)$, $x\in \Sp(\nu)$ and $X$ be an $(r,r)$-partition diagram. Concatenate $X$ and $v$ to get $\delta^t v'$ for some $(r,r-l)$-partition diagram $v'$ and some non-negative integer $t$. If $v'$ has fewer than $(r-l)$ propagating blocks then we set
$X(v\otimes x)=0$.
Otherwise we set $X(v\otimes x)=\delta^t v'\otimes x$. Note that in this case we have $v'=v''\sigma$ for a unique $v''\in V(r,r-l)$ with non-crossing propagating lines and a unique $\sigma\in \mathfrak{S}_{r-l}$ and we have $\delta^{t}v'\otimes x = \delta^tv''\otimes \sigma x$.

\subsection{Non-semisimple representation theory of the partition algebra}\label{comb}
We assume that $\delta=n\in\mathbb{Z}_{>0}$ (as otherwise the algebra is semisimple).

\begin{defn}Let $\mu\subset \lambda$ be partitions. 
We say that $(\mu,\lambda)$ is an $n$-pair, and write $\mu\hookrightarrow_n \lambda$, if the  Young diagram of $\lambda$ differs from the  Young diagram of $\mu$ by a horizontal row of boxes of which the last (rightmost) one has content $n-|\mu|$.
\end{defn}
\begin{eg}
For example, $((2,1), (4,1))$ is a $6$-pair.  We have that $6-|\mu|=3$ and the  Young diagrams (with contents) are as follows:
$$\young(01,\minusone) \ \subset \ \young(0123,\minusone)$$
note that they differ by $\young(23)$.
\end{eg}
Recall that the set of simple (or standard) modules for $P_r(n)$ are labelled by the set $\Lambda_{\leq r}$. This set splits into $P_r(n)$-blocks. The set of labels in each block forms a maximal chain of $n$-pairs
$$\lambda^{(0)}\hookrightarrow_n \lambda^{(1)} \hookrightarrow_n \lambda^{(2)} \hookrightarrow_n  \ldots \hookrightarrow_n \lambda^{(t)}.$$
Moreover, for $1\leq i\leq t$ we have that $\lambda^{(i)}/\lambda^{(i-1)}$ consists of a strip of boxes in the $i$th row.
Now we have an exact sequence of $P_r(n)$-modules
$$0\rightarrow \Delta_r( \lambda^{(t)})\rightarrow \ldots \rightarrow \Delta_r(\lambda^{(2)})\rightarrow \Delta_r(\lambda^{(1)})\rightarrow \Delta_r(\lambda^{(0)})\rightarrow L_r(\lambda^{(0)})\rightarrow 0$$
with the image of each homomorphism being simple. Each standard module $\Delta_r(\lambda^{(i)})$ (for $0\leq i \leq t-1$) has Loewy structure
$$\begin{array}{c} L_r(\lambda^{(i)})\\ L_r(\lambda^{(i+1)}) \end{array}$$
and so in the Grothendieck group we have
\begin{equation}\label{alt}
[L_r(\lambda^{(i)})]=\sum_{j=i}^{t} (-1)^{j-i} [\Delta_r(\lambda^{(j)})].
\end{equation}
Note that each block is totally ordered by the size of the partitions.

 \begin{prop}\label{singular}
Let $\nu\in \Lambda_{\leq r}$ and assume that $\nu_{[n]}$ is a partition. Then we have that
(i) $\nu$ is the minimal element in its $P_r(n)$-block, and \\
(ii) $\nu$ is the unique element in its block if and only if $n+1-\nu_1 > r$.
 \end{prop}
\begin{proof}
(i) Observe that for $\nu_{[n]}$ to be a partition we must have $n-|\nu|\geq \nu_1$. This implies that $\ct(\nu_1)=\nu_1 -1\leq n-|\nu| -1$. So we have $\nu \hookrightarrow_n \mu$ for some partition $\mu$ with $\mu/\nu$ being a single strip in the first row. Thus we have $\nu=\nu^{(0)}$ and $\mu=\nu^{(1)}$.\\
(ii) Now as $\nu^{(1)}/\nu$ is a single strip in the first row with last box having content $n-|\nu|$, we have that $|\nu^{(1)}/\nu|=n-|\nu|+1-\nu_1$ and thus $|\nu^{(1)}|=n+1-\nu_1$. Thus if $n+1-\nu_1>r$ then $\nu^{(1)}\notin \Lambda_{\leq r}$ and we have that $\nu$ is the only partition in its $P_r(n)$-block.
\end{proof}


  \section{Schur--Weyl duality } 
 Classical Schur--Weyl duality is the relationship between the general linear and symmetric groups over tensor space.  To be more specific, let $V_n$ be an $n$-dimensional complex vector space and let $V^{\otimes r}_n$ denote its $r$th tensor power.

We have that the symmetric group $\mathfrak{S}_r$ acts on the right by permuting the factors.  The general linear group, $\GL_n$, acts on the left by matrix multiplication on each factor.  These two actions commute and moreover $\GL_n$ and $\mathfrak{S}_r$ are full mutual centralisers in $\End(V^{\otimes r}_n)$.  

The partition algebra, $P_r(n)$, plays the role of the symmetric group, $\mathfrak{S}_r$, when we restrict the action of $\GL_n$ to the subgroup of permutation matrices, $\mathfrak{S}_n$.

\subsection{Schur-Weyl duality between $\mathfrak{S}_n$ and $P_r(n)$} 

 Let $V_n$ denote an $n$-dimensional complex space.  Then $\mathfrak{S}_n$ acts on $V_n$ via the permutation matrices. 
\begin{equation}
\label{Vaction}
\sigma\cdot v_i = v_{\sigma(i)}\qquad \mbox{ for } \sigma \in \mathfrak{S}_n.
\end{equation}
Notice that we are simply restricting the $\GL_n$ action in the classical Schur-Weyl duality to the permutation matrices.  Thus, $\mathfrak{S}_n$ acts diagonally on the basis of simple tensors of $V_n^{\otimes r}$ as follows
\[\sigma\cdot (v_{i_1}\otimes v_{i_2} \otimes \cdots \otimes v_{i_r}) = v_{\sigma(i_1)}\otimes v_{\sigma(i_2)} \otimes \cdots \otimes v_{\sigma(i_r)}.\]

\medskip
For each $(r,r)$-partition diagram $d$ and each integer sequence $i_1\ldots, i_r, i_{\bar 1}, \ldots, i_{\bar r}$ with $1\leq i_j, i_{\bar j}\leq n$, define
\begin{equation}
\label{eq:diagrammatrix}
\phi_{r,n}(d)_{i_{\bar 1}, \ldots, i_{\bar r}}^{i_1,\ldots, i_r} = \begin{cases} 1 & \mbox{ if $i_t= i_s$ whenever vertices $t$ and $s$ are connected in $d$} \\ 0 & \mbox{ otherwise.}
\end{cases}
\end{equation}
A partition diagram $d\in P_r(n)$ acts on the basis of simple tensors of $V_n^{\otimes r}$ as follows
\[\Phi_{r,n}(d) (v_{i_{1}} \otimes v_{i_{2}} \otimes \cdots \otimes v_{i_{r}}) = \sum_{i_{\bar 1}, \ldots i_{\bar r}} \phi_{r,n}(d)_{i_{\bar 1}, \ldots, i_{\bar r}}^{i_1,\ldots, i_r} v_{i_{\bar 1}} \otimes v_{i_{\bar 2}} \otimes \cdots \otimes v_{i_{\bar r}}.\]

\begin{thm}[Jones \cite{Jones}] \label{SWduality}
$\mathfrak{S}_n$ and $P_r(n)$ generate the full centralisers of each other in $\End(V_n^{\otimes r})$.  
\begin{enumerate}
\item[(a)] $P_r(n)$ generates $\End_{\mathfrak{S}_n}(V_n^{\otimes r})$, and when $n\geq 2r$,  $P_r(n) \cong \End_{\mathfrak{S}_n}(V_n^{\otimes r})$. 
\item[(b)] $\mathfrak{S}_n$ generates $\End_{\mathfrak{S}_n}(V_n^{\otimes r})$.

\end{enumerate}
\end{thm} 
 
\noindent We will denote $E_r(n)=\End_{\mathfrak{S}_n}(V_n^{\otimes r})$.

\begin{thm}[\cite{mar1} see also \cite{HalvRam}]We have a decomposition of $V_n^{\otimes r}$ as a $(\mathfrak{S}_n,P_r(n))$-bimodule
$$V_n^{\otimes r} = \bigoplus   \Sp(\lambda_{[n]})\otimes  L_r(\lambda)$$
where the sum is over all  partitions $\lambda_{[n]}$ of $n$ such that $|\lambda| \leq r$. 
\end{thm}
 
Using \cite[Theorem 9.2.2]{goodwall} we have, for $\lambda_{[n]}, \mu_{[n]},\nu_{[n]} \vdash n$ with $\lambda \vdash r$ and $\mu \vdash s$,
\begin{equation}\label{gw}\Hom_{\mathfrak{S}_n}(\Sp(\nu_{[n]}), \Sp(\lambda_{[n]})\otimes \Sp(\mu_{[n]}))\end{equation}
\begin{eqnarray*} 
 && \cong  \left\{ 
\begin{array}{ll} \Hom_{E_r(n)\otimes E_s(n)}(L_r(\lambda) \boxtimes L_s(\mu), L_{r+s}(\nu)\!\downarrow_{E_r(n)\otimes E_s(n)}) & \mbox{if $\nu \in \Lambda_{\leq r+s}$}\\
0 & \mbox{otherwise.}
\end{array}\right.
\end{eqnarray*}

\subsection{Kronecker product via the partition algebra}\label{3.2}
Going back to the formula in (\ref{gw}) we need to consider $L_{r+s}(\nu)\!\downarrow_{E_r(n)\otimes E_s(n)}$. Now $L_{r+s}(\nu )$ is a simple $P_{r+s}(n)$-module annihilated by $\ker \Phi_{r+s,n}$ and hence also by $\ker \Phi_{r,n} \otimes \ker \Phi_{s,n}$. Thus $L_{r+s}(\nu )\!\!\downarrow_{P_r(s)\otimes P_s(n)}$ is semisimple and has the same simple factors as $L_{r+s}(\nu)\!\!\downarrow_{E_r(s)\otimes E_s(n)}$.

Now combining (\ref{gw}) with (\ref{alt}) we have the following result.

\begin{thm}\label{transfer}
Let $\lambda_{[n]}, \mu_{[n]}, \nu_{[n]}\vdash n$ with $\lambda \vdash r$ and $\mu \vdash s$. Then we have
$$
{g}_{\lambda_{[n]} , \mu_{[n]} }^{\nu_{[n]}} =   
\left\{
\begin{array}{ll} \sum\limits_{i=0}^t (-1)^i [\Delta_{r+s}(\nu^{(i)})\!\downarrow_{P_r(n)\otimes P_s(n)}:L_r(\lambda)\boxtimes L_s(\mu)] & \mbox{if $\nu\in \Lambda_{\leq(r+s)}$}\\ 0 & \mbox{otherwise} \end{array} 
\right.
$$
where $\nu  = \nu^{(0)} \hookrightarrow_n \nu^{(1)} \hookrightarrow_n \ldots \hookrightarrow \nu^{(t)}$ is the $P_{r+s}(n)$-block of $\nu$. 
\end{thm}
For sufficiently large values of $n$ the partition algebra is semisimple.  Therefore Theorem \ref{transfer} reproves the limiting behaviour of tensor products observed by Murnaghan (see Section \ref{murg}).   It also  offers the following concrete representation theoretic interpretation of the $\overline{g}_{\lambda , \mu}^{\nu}$. 
 
 \begin{cor}
 Let  $\lambda \vdash r$ and $\mu \vdash s$ and suppose $|\nu|\leq r+s$.  Then we have
$$
 \overline{g}_{\lambda , \mu }^{\nu} =   [\Delta_{r+s}(\nu)\!\downarrow_{P_r(n)\otimes P_s(n)}:L_r(\lambda)\boxtimes L_s(\mu)] .
$$
 \end{cor}
 
  \begin{rmk}We recover the Murnaghan--Littlewood Theorem as follows.
Let $\lambda,\mu,\nu$ be partitions and suppose that $|\lambda|+|\mu|=|\nu|$.  Then we have that $\Delta_{r+s}(\nu)=\Sp(\nu)$, $\Delta_{r}(\lambda)=\Sp(\lambda)$ and $\Delta_s(\mu)=\Sp(\mu)$  and so we have
$$\overline{g}^\nu_{\lambda,\mu}=c^\nu_{\lambda,\mu}.$$
\end{rmk}

\begin{cor} \label{corol} 
We have that $\overline{g}_{\lambda,\mu}^\nu={g}_{	\lambda_{[n]},\mu_{[n]}	}^{	\nu_{[n]}	}$ if 
$$n \geq {\rm min}\{|\lambda|+|\mu|+\nu_1,|\lambda|+|\nu|+\mu_1, |\nu|+|\mu|+\lambda_1\}.$$
\end{cor}
\begin{proof}
 When $n\geq |\lambda|+|\mu|+\nu_1$ we have that $\Delta_{r+s}(\nu)=L_{r+s}(\nu)$ by Proposition \ref{singular}.
  The result now follows as 
 \begin{equation*}g^{\nu_{[n]}}_{\lambda_{[n]}, \mu_{[n]}}=
  g^{\mu_{[n]}}_{\lambda_{[n]}, \nu_{[n]}}=
   g^{\lambda_{[n]}}_{\nu_{[n]}, \mu_{[n]}}.\qedhere \end{equation*}
\end{proof}

Corollary \ref{corol}  gives a new proof of  Brion's bound \cite{brion} for the stability of the Kronecker coefficients using the partition algebra. 

\subsection{The Kronecker coefficients as a sum of reduced Kronecker coefficients}\label{3.3}
In \cite{rosa} a formula is given for writing the Kronecker coefficients as a sum of reduced Kronecker coefficients.  We shall now interpret this formula in the Grothendieck group of the partition algebra by showing that it coincides with the formula  in Theorem \ref{transfer}.

Let $\nu_{[n]}$ be a partition of $n$.  We make the convention that $\nu_0=n-|\nu|$ is the $0$th row of $\nu_{[n]}$. For $i\in \mathbb{Z}_{\geq 0}$ define ${\nu_{[n]}^{\dagger i}}$ to be the partition obtained from $\nu_{[n]}$ by adding 1 to its first $i-1$ rows and erasing its $i$th row. In particular we have ${\nu_{[n]}^{\dagger 0}}=\nu$. 

\begin{thm}[Theorem 1.1 of \cite{rosa}]\label{rossss}
Let $\lambda_{[n]}, \mu_{[n]}, \nu_{[n]}\vdash n$. 
Then 
$$g^{\nu_{[n]}}_{\lambda_{[n]},\mu_{[n]}}	= \sum_{i=0}^{l} (-1)^i \overline{g}^{\nu_{[n]}^{\dagger i}}_{\lambda,\mu}	 			$$
where $l=\ell(\lambda_{[n]})\ell(\mu_{[n]})-1$.
\end{thm}

Relating this to the partition algebra, we have the following.

\begin{prop}
Let $  \nu_{[n]}\vdash n$ and let $ \nu= \nu^{(0)} \hookrightarrow_n \nu^{(1)}\hookrightarrow_n \ldots $ be a chain of $n$-pairs.  Then the partitions
$${\nu_{[n]}^{\dagger i}}= \nu^{(i)}$$
for all $ i \geq 0$.   
\end{prop}

\begin{proof}
The $i=0$ case is clear from the definitions.  We proceed by induction.  Assume that
$${\nu_{[n]}^{\dagger i}}= \nu^{(i)}.$$
Then $(  \nu^{(i)} )_1= n- |\nu| + 1$,  $(  \nu^{(i)} )_j=  \nu_{j-1} + 1$ for $j\leq i$, and  $(  \nu^{(i)} )_j= \nu_{j}$ for $j>i$.  Therefore 
\begin{align*}
|\nu^{(i)}|=n- |\nu|+1 +\sum_{j \neq i}\nu_j + i-1 = n- \nu_i+i
\end{align*}

We have that $   \nu^{(i+1)} /  \nu^{(i)}$ is a skew partition consisting of a strip in the $(i+1)$th row.  By definition of an $n$-pair the content, $\ct(\nu^{(i+1)}_{i+1})$, of the last node is $n-| \nu^{(i)}|$.  Therefore
\begin{align*}
\ct(\nu^{(i+1)}_{i+1}):= \nu^{(i+1)}_{i+1} - (i+1)=  n- (n- \nu_i+i) = \nu_i - i
\end{align*}
and $\nu^{(i+1)}_{i+1}=\nu_i+1$, therefore ${\nu_{[n]}^{\dagger (i+1)}}= \nu^{(i+1)}$.
\end{proof}

\begin{rmk}
In Theorem \ref{transfer}, $t$ is chosen so that $|\nu^{(t)}| \leq |\lambda|+|\mu|$ and $|\nu^{(t+1)}| > |\lambda|+|\mu|$.  So Theorem \ref{transfer} and \ref{rossss} seem to give a different number of terms in the sum.  For example consider 
$$g^{(2)}_{(1^2),(1^2)}=1  \quad g^{(1^2)}_{(1^2),(1^2)}=0$$ 
these are given as a sum of one, respectively two terms in Theorem \ref{transfer}, both cases have four terms in Theorem \ref{rossss}.  Now consider $$\lambda_{[n]}=\mu_{[n]}=\nu_{[n]}=(10,10,10)$$ then $\ell(\lambda_{[n]})\ell(\mu_{[n]})=9$. We have $\nu_{[n]}^{\dagger 8}=(11^3,1^5)$ with $|\nu_{[n]}^{\dagger 8}|=38$. But $r+s=40$, so we have two more terms in Theorem \ref{transfer}, corresponding to $\nu^{(9)}=(11^3, 1^6)$ and $\nu^{(10)}=(11^3, 1^7)$. However, we can show that in fact the two theorems give the same sum.

First assume that $\ell(\lambda_{[n]})\ell(\mu_{[n]})-1>t$, then for all $i>t$ we have $$\bar{g}_{\lambda , \mu}^{\nu_{[n]}^{\dagger i}}=0$$ as $|\nu_{[n]}^{\dagger i}|>|\lambda|+|\mu|$. And so the two sums coincide.

Now assume that $\ell(\lambda_{[n]})\ell(\mu_{[n]})-1 <t$. Then for all $i>\ell(\lambda_{[n]})\ell(\mu_{[n]})-1$, we have 
$$i\geq \ell(\lambda_{[n]})\ell(\mu_{[n]}) \geq |\lambda_{[n]}\cap (\mu_{[n]})'|.$$
where $(\mu_{[n]})'$ denotes the conjugate partition of $\mu_{[n]}$.  (To see this observe that the Young diagram of $\lambda_{[n]}\cap (\mu_{[n]})'$ fits in a rectangle of size $\ell(\lambda_{[n]}) \times \ell(\mu_{[n]})$).
Now we have
$$\ell(\nu^{(i)})\geq i \geq |\lambda_{[n]}\cap (\mu_{[n]})'|.$$
But this implies that $\bar{g}_{\lambda, \mu}^{\nu^{(i)}}=0$ by \cite{Dvir}.
\end{rmk}


\section{The restriction of a standard module to a  Young subalgebra }
In this section we compute the restriction of a standard module to a Young subalgebra of the partition algebra.
  
Set $m=r+s$ for some $r,s\geq 1$ and fix $\delta\in \mathbb{C}^\times$.
We write $P_r=P_r(\delta)$, $P_s=P_s(\delta)$ and $P_m=P_m(\delta)$. We view $P_r\otimes P_s$ as a subalgebra of $P_m$ by mapping each $d\otimes d'$, where $d$ (resp. $d'$) is an $(r,r)$- (resp. $(s,s)$-) partition diagram, to the $(m,m)$-partition diagram obtained by putting $d$ and $d'$ side by side, with $d$ to the left of $d'$.

We wish to understand the restriction of $\Delta_m(\nu)$ to the subalgebra $P_r\otimes P_s$. Let $\mathfrak{D}_r$ denote the diagonal copy of $\mathfrak{S}_r$ in $\mathfrak{S}_r\times \mathfrak{S}_r$.  We will need the following lemmas.

\begin{lem}\label{brauerlemma} Let $V(2r,0)_r$ be the subspace of $V(2r,0)$ spanned by all $(2r,0)$-partition diagrams having precisely $r$ blocks of the form $\{i,j\}$ with $1\leq i\leq r$ and $r+1\leq j\leq 2r$. Then $V(2r,0)_r$ is a $\mathfrak{S}_r\times \mathfrak{S}_r$-module and we have
$$V(2r,0)_r \cong \mathbb{C}\!\uparrow_{\mathfrak{D}_r}^{\mathfrak{S}_r \times \mathfrak{S}_r} \cong \bigoplus_{\lambda\vdash r} \Sp(\lambda)\boxtimes \Sp(\lambda).$$
The structure as a $P_r\times P_r$-module is (trivially) obtained by inflation.
\end{lem}

\begin{proof}
A basis for $V(2r,0)_r$ is given by the set $\{(1,\sigma)v_0 : \sigma \in \mathfrak{S}_r\}$ where $v_0$ is the diagram given in Figure \ref{fig1}.
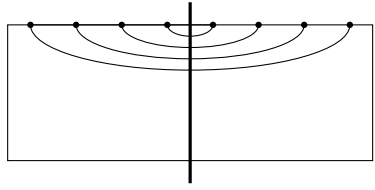
\begin{figure}[ht]
\begin{tikzpicture}[scale=0.6]
  \draw (0,0) rectangle (8,3);
  \foreach \x in {0.5,1.5,...,7.5}
    {\fill (\x,3) circle (2pt);
    }
  \begin{scope}[very thick]
        \draw (4,3.5) -- (4,-.5);
        \end{scope}
  \begin{scope}
    \draw (0.5,3) -- (4.5,3);
    \draw (2.5,3) arc (180:360:1.5 and .5);
        \draw (3.5,3) arc (180:360:.5 and .25);
    \draw (0.5,3) arc (180:360:3.5 and 1);
        \draw (1.5,3) arc (180:360:2.5 and 0.75);
  \end{scope}
\end{tikzpicture}

 \caption{The diagram $v_0$ in $V(2r,0)_r$}
\label{fig1}
\end{figure}

Now, the map 
$$f: V(2r,0)_r \to \mathbb{C}\!\uparrow_{\mathfrak{D}_r}^{\mathfrak{S}_r\times \mathfrak{S}_r}$$
given by 
$$f((1,\sigma)v_0)=(1,\sigma)$$
gives the required isomorphism.
%
\end{proof}

\begin{lem}\label{partitionlemma} Let $V(2r,r)_r$ be the subspace of $V(2r,r)$ spanned by all $(2r,r)$-partition diagrams having precisely $r$ propagating blocks of the form $\{i,j,\bar{k}\}$ with $1\leq i\leq r$ and $r+1\leq j\leq 2r$. Then, for any $\mu\vdash r$ we have that $V(2r,r)_r\otimes_{\mathfrak{S}_r}\Sp(\mu)$ is a $\mathfrak{S}_r\times \mathfrak{S}_r$-module and we have
$$V(2r,r)_r\otimes_{\mathfrak{S}_r}\Sp(\pi)\cong \mathbf{S}(\pi) \!\uparrow_{\mathfrak{D}_r}^{\mathfrak{S}_r\times \mathfrak{S}_r}\cong\bigoplus_{\rho,\sigma} g_{\rho,\sigma}^\pi \Sp(\rho)\boxtimes \Sp(\sigma).$$
The structure as a $P_r\times P_r$-module is (trivially) obtained by inflation.
\end{lem}

\begin{proof}
Let $X(\pi)$ be a basis for $\mathbf{S}(\pi)$.
A basis for $V(2r,r)_r\otimes_{\mathfrak{S}_r} \mathbf{S}(\pi)$ is given by the set 
$$\{(1,\sigma)v_1\otimes x : \sigma \in \mathfrak{S}_r, x \in X(\pi)					\}	$$
where $v_1$ is given in Figure \ref{fig2}.

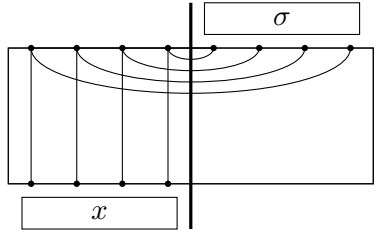
\begin{figure}[ht]
\begin{minipage}{2in}
\begin{tikzpicture}[scale=0.6]
  \draw (0,0) rectangle (8,3);
  \foreach \x in {0.5,1.5,...,7.5}
    {\fill (\x,3) circle (2pt);
    }
  \draw (0,0) rectangle (8,3);
  \foreach \x in {0.5,1.5,...,3.5}
    { 
      \fill (\x,0) circle (2pt);
    }
  \begin{scope}[very thick]
        \draw (4,4) -- (4,-1);
        \end{scope}
  \begin{scope}
    \draw (0.5,3) -- (0.5,0);
        \draw (1.5,3) -- (1.5,0);
            \draw (2.5,3) -- (2.5,0);
                \draw (3.5,3) -- (3.5,0);
    \draw (0.5,3) -- (4.5,3);
    \draw (2.5,3) arc (180:360:1.5 and .5);
        \draw (3.5,3) arc (180:360:.5 and .25);
    \draw (0.5,3) arc (180:360:3.5 and 1);
        \draw (1.5,3) arc (180:360:2.5 and 0.75);
          \draw (0.3,-.3) rectangle (3.7,-1);
          \draw (4.3,3.3) rectangle (7.7,4);
              \draw (2,-0.65) node {$  x$};
              \draw (6,3.65) node {$\sigma $};
  \end{scope}
\end{tikzpicture}
\end{minipage}
 \caption{The element $(1,\sigma)v_1 \otimes x$  }
\label{fig2}
\end{figure}

 Now the map 
$$g: V(2r,r)_r \otimes_{\mathfrak{S}_r} \mathbf{S}(\pi) \to \mathbf{S}(\pi)\!\uparrow_{\mathfrak{D}_r}^{\mathfrak{S}_r\times \mathfrak{S}_r}$$
given by
$$g(1,\sigma)v_1 \otimes x = (1,\sigma)\otimes x			$$
gives the required isomorphism.
\end{proof}

\begin{thm}\label{partition}
Write $m=r+s$ and let $\nu\vdash m-l$, $\lambda\vdash r-l_r$ and $\mu\vdash s-l_s$ for some non-negative integers $l, l_r, l_s$. Then $\Delta_m(\nu)\!\!\downarrow_{P_r\otimes P_s}$ has  a filtration by standard modules with multiplicities given by
$$[\Delta_m(\nu)\!\downarrow_{P_r\otimes P_s} \, : \, \Delta_r(\lambda)\boxtimes \Delta_s(\mu)]=\sum_{\begin{subarray}{c} l_1, l_2 \\ l_1+2l_2=l-l_r-l_s \end{subarray}} \sum_{\begin{subarray}{c} \alpha\vdash r-l_r-l_1-l_2 \\ \beta\vdash s-l_s-l_1-l_2 \\ \pi,\rho,\sigma\vdash l_1 \\ \gamma\vdash l_2 \end{subarray}} c_{\alpha ,\beta ,\pi}^\nu c_{\alpha , \rho ,\gamma}^\lambda c_{\gamma ,\sigma , \beta}^{\mu}g_{\rho,\sigma}^\pi .$$
\end{thm} 

\begin{rmk}
Note that these multiplicities are well-defined, this follows by the general theory of quasi-hereditary algebras, see \cite{dr1}.
\end{rmk}

\begin{proof} 
Recall that   $\Delta_m(\nu)=V(m,m-l)\otimes_{\mathfrak{S}_{m-l}}\Sp(\nu)$. Now, we say that a block of a diagram in $V(m,m-l)$ is a crossing block if it contains at least one vertex in $\{1, \ldots, r\}$ and one vertex in $\{r+1, \ldots , r+s\}$. It is clear that when we apply a diagram in $P_r\otimes P_s$ to a diagram in $V(m,m-l)$ we cannot increase the number of crossing blocks. Thus we get a filtration of $\Delta_m(\nu)\!\downarrow_{P_r\otimes P_s}$ with subquotients isomorphic to 
$$V(m,m-l)_c\otimes_{\mathfrak{S}_{m-l}} \Sp(\nu)$$
where $V(m,m-l)_c$ denotes the span of all diagrams in $V(m,m-l)$ having precisely $c$ crossing blocks. Now each subquotient splits as
$$V(m,m-l)_c\otimes_{\mathfrak{S}_{m-l}}\Sp(\nu) = \bigoplus_{\begin{subarray}{c} p_r, p_s, p_c, n_c \\ p_r+p_s+p_c=m-l \\ p_c+n_c=c\end{subarray}} V(m,m-l)_{p_r, p_s, p_c, n_c}\otimes_{\mathfrak{S}_{m-l}}\Sp(\nu)$$
where $V(m,m-l)_{p_r, p_s, p_c, n_c}$ is the span of all diagrams in $V(m,m-l)_c$ having precisely $p_r$ (resp. $p_s$) propagating blocks containing only vertices in the set $\{1, \ldots , r\}$ (resp. $\{r+1. \ldots , r+s\}$)  and some $\bar{j}$ (for some $1\leq j\leq m-l$), precisely $p_c$ propagating crossing blocks and precisely $n_c$ non-propagating crossing blocks.  This is illustrated in Figure \ref{fig4}.

\begin{figure}[ht]
 \begin{tikzpicture}[scale=0.6]
  \draw (0,0) rectangle (16,3);
  \foreach \x in {0.5,1.5,...,15.5}
    {\fill (\x,3) circle (2pt); }
    {\fill (2,0) circle (2pt); }
    {\fill (5,0) circle (2pt); }
    {\fill (8,0) circle (2pt); }
        {\fill (11,0) circle (2pt); }
                {\fill (14,0) circle (2pt); }
  \begin{scope}[very thick]
        \draw (10,3.5) -- (10,.7);
        \end{scope}
  \begin{scope}
    \draw (1.5,3) -- (2,0);
    \draw (7.5,3) -- (5,0);
    \draw (9.5,3) -- (8,0);
        \draw (15.5,3) -- (11,0);
            \draw (14.5,3) -- (14,0);
      \draw (1.5,3) arc (180:360:.5 and .3);
     \draw (9.5,3) arc (180:360:.5 and .5);
     \draw (8.5,3) arc (180:360:2.5 and 2);
     \draw (7.5,3) arc (180:360:.5 and .5);
     \draw (2.5,3) arc (180:360:.5 and .5);
     \draw (4.5,3) arc (180:360:3.5 and 1.5);
     \draw (5.5,3) arc (180:360:3.5 and 1);
  \end{scope}
\end{tikzpicture}
  \caption{An element, $w$, of $V(16,5)$ with $p_r$=1, $p_s=2$, $p_c=2$ and $n_c=2$}
\label{fig4}
\end{figure}
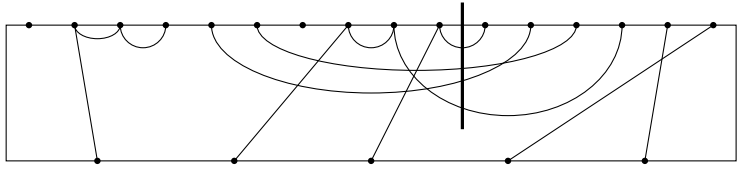

To see this, note that multiplication by elements of $P_r\otimes P_s$ on this subquotient preserves both the total number of propagating blocks and the total number of crossing blocks.

Now we have
\begin{align*}
&V(m,m-l)_{p_r,p_s,p_c,n_c}\otimes_{\mathfrak{S}_{m-l}}\Sp(\nu) \\
=& V(m,m-l)_{p_r,p_s,p_c,n_c}\otimes_{\mathfrak{S}_{p_r}\times \mathfrak{S}_{p_c}\times \mathfrak{S}_{p_s}}\Sp(\nu)\!\downarrow_{\mathfrak{S}_{p_r}\times \mathfrak{S}_{p_c}\times \mathfrak{S}_{p_s}}\\
=& V(m,m-l)_{p_r,p_s,p_c,n_c}\otimes_{\mathfrak{S}_{p_r}\times \mathfrak{S}_{p_c}\times \mathfrak{S}_{p_s}} \bigoplus_{\begin{subarray}{c} \alpha\vdash p_r\\\beta\vdash p_s \\ \pi\vdash p_c\end{subarray}} c_{\alpha , \beta ,\pi}^\nu \Sp(\alpha)\boxtimes \Sp(\pi)\boxtimes \Sp(\beta).
\end{align*}
The key point of the proof is that 
\begin{equation*}V(m,m-l)_{p_r, p_s, p_c, n_c}\otimes_{\mathfrak{S}_{p_r, p_c, p_s}} \Sp(\alpha)\boxtimes \Sp(\pi)\boxtimes \Sp(\beta) \tag{$\dagger$}   \label{dagger}  
\end{equation*} is isomorphic, as a $P_r \times P_s$-module, to
\begin{equation*}  
 (V(r, p_r+p_c+n_c)\boxtimes V(s, p_s+p_c+n_c)) \otimes_{\mathfrak{S}_{p_r+p_c+n_c, p_s+p_c+n_c}}
\end{equation*}
$$ (\Sp(\alpha) \boxtimes (V(2p_c, p_c)_{p_c}\otimes_{\mathfrak{S}_{p_c}}\Sp(\pi)) \boxtimes V(2n_c, 0)_{n_c} \boxtimes \Sp(\beta))\!\uparrow_{\mathfrak{S}_{p_r,p_c,n_c,n_c,p_c,p_s}}^{\mathfrak{S}_{p_r+p_c+n_c, n_c+p_c+p_s}}.$$
This follows by `decomposing the diagrams', this is best illustrated by an example.
Under this isomorphism, the element $w \otimes (x_1 \otimes x_2 \otimes x_3)$ where $w$ is as in Figure \ref{fig4} and $x_i \in \Sp(\alpha)$ is mapped to the element in Figure \ref{fig5}.

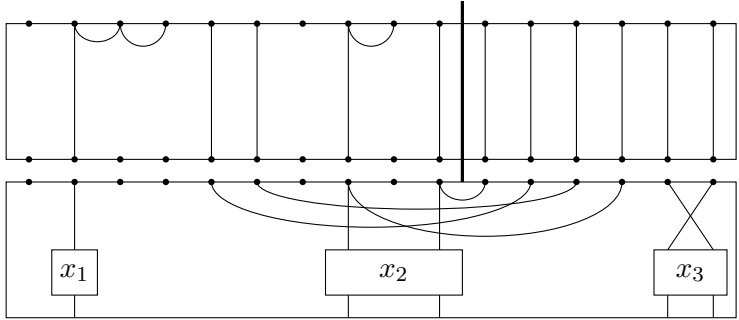
\begin{figure}[ht]
 \begin{tikzpicture}[scale=0.6]
  \draw (0,0) rectangle (16,3);
   \draw (0,-.5) rectangle (16,-3.5);
  \foreach \x in {0.5,1.5,...,15.5}
    {\fill (\x,3) circle (2pt); \fill (\x,0) circle (2pt); \fill (\x,-0.5) circle (2pt); }
  \begin{scope}[very thick]
        \draw (10,3.5) -- (10,-0.5);
        \end{scope}
  \begin{scope}
    \draw (1.5,3) -- (1.5,0);
    \draw (10.5,3) -- (10.5,0);
    \draw (13.5,3) -- (13.5,0);
    \draw (7.5,3) -- (7.5,0);
    \draw (9.5,3) -- (9.5,0);
        \draw (15.5,3) -- (15.5,0);
            \draw (14.5,3) -- (14.5,0);
            \draw (1.5,-2) -- (1.5,-0.5);
                        \draw (1.5,-3) -- (1.5,-3.5);
            \draw (7.5,-2) -- (7.5,-0.5);
                        \draw (7.5,-3) -- (7.5,-3.5);
            \draw (9.5,-2) -- (9.5,-0.5);
                        \draw (9.5,-3) -- (9.5,-3.5);
         \draw (4.5,3) -- (4.5,0); 
         \draw (14.5,-0.5) -- (15.5,-2); 
         \draw (14.5,-3) -- (14.5,-3.5);          \draw (15.5,-3) -- (15.5,-3.5); 
         \draw (15.5,-0.5) -- (14.5,-2); 
         \draw (5.5,3) -- (5.5,0); 
                  \draw (11.5,3) -- (11.5,0);          \draw (12.5,3) -- (12.5,0); 
      \draw (1.5,3) arc (180:360:.5 and .4);
     \draw (9.5,-.5) arc (180:360:.5 and .4);
     \draw (7.5,-.5) arc (180:360:3 and 1.2);
     \draw (7.5,3) arc (180:360:.5 and .5);
     \draw (2.5,3) arc (180:360:.5 and .5);
     \draw (4.5,-.5) arc (180:360:3.5 and 1);
     \draw (5.5,-.5) arc (180:360:3.5 and 0.6);
                   \draw (1,-2) rectangle (2,-3);
                      \draw (7,-2) rectangle (10,-3);
                      \draw (14.2,-2) rectangle (15.8,-3);
                                     \draw (1.5,-2.5) node {$ x_1$};
                                      \draw (8.5,-2.5) node {$ x_2$};
                                       \draw (15,-2.5) node {$ x_3$};
  \end{scope}
\end{tikzpicture}
  \caption{The image of $w \otimes (x_1 \otimes x_2 \otimes x_3)$ under the `decomposing diagrams' isomorphism}
\label{fig5}
\end{figure}

By  application of Lemmas \ref{brauerlemma} and \ref{partitionlemma}, equation (\ref{dagger}) is isomorphic to
$$   (V(r, p_r+p_c+n_c)\boxtimes V(s, p_s+p_c+n_c)) \otimes_{\mathfrak{S}_{p_r+p_c+n_c, p_s+p_c+n_c}}$$
$$\left(\oplus_{\gamma\vdash n_c}\oplus_{\rho,\sigma\vdash  p_c}g_{\rho,\sigma}^\pi
\Sp(\alpha) \boxtimes \Sp(\rho) \boxtimes \Sp(\gamma)\boxtimes \Sp(\gamma)\boxtimes \Sp(\sigma) \boxtimes \Sp(\beta)\right)\! {\large \uparrow}_{\mathfrak{S}_{p_r,p_c,n_c,n_c,p_c,p_s}}^{\mathfrak{S}_{p_r+p_c+n_c, n_c+p_c+p_s}}$$
which by the Littlewood--Richardson rule is isomorphic to
$$\oplus_{\gamma\vdash n_c}\oplus_{\rho,\sigma\vdash  p_c}g_{\rho,\sigma}^\pi (V(r, p_r+p_c+n_c)\boxtimes V(s, p_s+p_c+n_c)) \otimes_{\mathfrak{S}_{p_r+p_c+n_c, p_s+p_c+n_c}}$$
$$\oplus_{\begin{subarray}{c} \lambda\vdash p_r+p_c +n_c \\ \mu\vdash p_s+p_c+n_c \end{subarray}} c_{\alpha , \rho ,\gamma}^\lambda c_{\gamma ,\sigma ,\beta}^\mu \Sp(\lambda)\boxtimes \Sp(\mu).$$
We can now rewrite this as a product of standard modules as follows:
$$ \bigoplus_{\begin{subarray}{c} {\gamma\vdash n_c} \\ {\rho,\sigma\vdash  p_c} \end{subarray}}
\bigoplus_{\begin{subarray}{c} \lambda\vdash p_r+p_c+n_c \\ \mu\vdash p_s+p_c+n_c \end{subarray}}  g_{\rho,\sigma}^\pi c_{\alpha , \rho ,\gamma}^\lambda c_{\gamma , \sigma ,\beta}^\mu\Delta_r(\lambda)\boxtimes \Delta_s(\mu).$$
Now noting that $p_r+p_s+p_c=m-l$, $p_r+p_c+n_c=r-l_r$ and $p_s+p_c+n_c=s-l_s$ and writing $l_1=p_c$ and $l_2=n_c$, we get that $l-l_r-l_s=l_1+l_2$ and the result follows.
\end{proof}

 In \cite[Lemma 2.1]{rosa} a formula is given for writing the reduced Kronecker coefficients as a sum of Kronecker coefficients and Littlewood--Richardson coefficients.  An immediate corollary of the above theorem is an interpretation of this formula in the setting of the partition algebra.

\begin{cor}
\label{bigthm}
Let $\lambda,\mu,\nu$ be any partitions with $|\lambda|=r$, $|\mu|=s$ and $|\nu|=r+s-l$. Then the reduced Kronecker coefficient $\overline{g}_{\lambda , \mu}^{\nu}$ is given by 
$$
\overline{g}^\nu_{\lambda,\mu}=
 \sum_{\begin{subarray}{c}l_1, l_2 \\
l=l_1+2l_2\end{subarray}}
  \sum_{\begin{subarray}{c}{\alpha  \vdash r-l_1-l_2}
    \\  {\beta\vdash s-l_1-l_2}  \end{subarray}}
      \sum_{\begin{subarray}{c}
 {\pi,\rho,\sigma \vdash l_1} \\{\gamma\vdash l_2}
    \end{subarray}} 
    c_{\alpha ,\beta,\pi}^{\nu } c_{\alpha ,\rho,\gamma}^{\lambda } c_{ \gamma,  \sigma, \beta}^{\mu } g_{\rho,\sigma}^\pi   $$
\end{cor}
\begin{proof}
This follows from Theorems \ref{transfer} and \ref{partition}, noting that for $|\lambda|=r$ and $|\mu|=s$, $\Delta_r(\lambda)=L_r(\lambda)$ and $\Delta_r(\mu)=L_r(\mu)$.
\end{proof}


\section{Hooks and two-part partitions}

We now consider the case where one of the partitions in a Kronecker coefficient is either a hook or two-part partition. 
The first positive closed formula for the two-part partition case was due to Ballantine and Orellana \cite{BallantineOrellanaEJC}.  Blasiak \cite{blasiak} has recently given a combinatorial interpretation of the one hook case.  

The result below provides positive closed formulas for $g_{\lambda[n],\mu[n]}^{\nu_{[n]}}$ in the case that $\nu_{[n]}$ is a two-part or hook partition; in the hook case, this is the first such closed formula, in the two-part case our formula simplifies that of \cite{BallantineOrellanaEJC}.
These formulas reveal a distinct symmetry between the two cases.
\begin{cor} 
\label{bigthmspecialcase2}
Let $\lambda_{[n]}, \mu_{[n]}, \nu_{[n]}$ be partitions of $n$ with $|\lambda|=r$, $|\mu|=s$ and  $|\nu|=r+s-l$.  

(i) Suppose  $\nu_{[n]}=(n-k,k)$ is a two-part partition.  
Then we have
$$
{g}^{(n-k,k)}_{\lambda_{[n]},\mu_{[n]}}=
 \sum_{\begin{subarray}{c}l_1, l_2 \\
l=l_1+2l_2\end{subarray}}
      \sum_{\begin{subarray}{c}
 { \sigma \vdash l_1} \\{\gamma\vdash l_2}
    \end{subarray}} 
 c_{(r-l_1-l_2) ,\sigma,\gamma}^{\lambda } c_{ \gamma,  \sigma,(s-l_1-l_2)}^{\mu }   $$
for all $n \geq {\rm min}\{|\lambda|+\mu_1+k, |\mu|+\lambda_1+k\}$.

 (ii) Suppose  $\nu_{[n]}=(n-k,1^k)$ is a hook partition.
Then we have
$$
{g}^{(n-k,1^k)}_{\lambda_{[n]},\mu_{[n]}}=
 \sum_{\begin{subarray}{c}l_1, l_2 \\
l=l_1+2l_2\end{subarray}}
      \sum_{\begin{subarray}{c}
 { \sigma \vdash l_1} \\{\gamma\vdash l_2}
    \end{subarray}} 
 c_{(1^{r-l_1-l_2  }),\sigma,\gamma}^{\lambda } c_{ \gamma,  \sigma',(1^{s-l_1-l_2})}^{\mu }   $$
for all $n \geq {\rm min}\{|\lambda|+|\mu| +1, |\mu|+\lambda_1+k, |\lambda|+\mu_1+k\}$ 
and where $\sigma'$ denotes the transpose of $\sigma$.
\end{cor}
\begin{proof}
Our assumption on $n$ implies that ${g}^{\nu_{[n]}}_{\lambda_{[n]},\mu_{[n]}}= \overline{g}^\nu_{\lambda,\mu}$ by Corollary \ref{corol}.

The result follows from Corollary \ref{bigthm}, noting that 
$c_{\alpha,\beta,\pi}^{(k)}$ (respectively $c_{\alpha,\beta,\pi}^{(1^k)}$) is zero unless $\alpha=(r-l_1-l_2)$, $\beta=(s-l_1-l_2), \pi = (l_1)$ (respectively $\alpha=(1^{r-l_1-l_2})$, $\beta=(1^{s-l_1-l_2}), \pi = (1^{l_1})$) in which case it is equal to 1
 and $g^{(l_1)}_{\rho,\sigma}$ (respectively $g^{(1^{l_1})}_{\rho,\sigma}$) is zero unless  $\rho=\sigma$ (respectively  $\rho=\sigma'$), in which case it is equal to 1.
\end{proof}

\begin{rmk}
In \cite{BallantineOrellanaEJC} they compute the Kronecker coefficients 
$$g_{(n-k,k),\lambda_{[n]}}^{\mu_{[n]}} = g_{\lambda_{[n]},\mu_{[n]}}^{(n-k,k)}$$ 
when $n-|\lambda|-\lambda_1\geq 2k$, equivalently
$$n \geq |\lambda|+\lambda_1+2k.$$
Noting that $k=|\mu|$ and for $\overline{g}^\nu_{\lambda,\mu}\neq0$, we must have that
 $|\mu|\leq |\lambda|+|\nu|$,
 we see that Corollary \ref{bigthmspecialcase2} improves this bound (as $|\mu|+\lambda_1+k \leq  |\lambda|+\lambda_1+2k)$.
\end{rmk}

\section{Example }

In this section, we shall  compute the tensor square of the Specht module, $\Sp(n-1,1)$ for $n\geq 2$, labelled by the first non-trivial hook, via the partition algebra.  We have that 
$$ \Hom_{\mathfrak{S}_n}(\Sp(\nu_{[n]}), \Sp(n-1,1)\otimes \Sp(n-1,1)) 
 \cong\Hom_{P_1(n)\otimes P_1(n)}(L_1(1) \otimes L_1(1), L_{2}(\nu)\!\!\downarrow)
 $$ if $\nu \in\Lambda_{\leq 2}$ and zero otherwise.
Therefore, it is enough to consider the restriction of simple modules from $P_2(n)$ to the Young subalgebra $P_1(n) \otimes P_1(n)$.

The partition algebra $P_2(n)$ is a 15-dimensional algebra with basis:
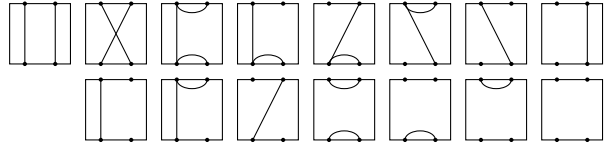
\begin{align*}
  \begin{minipage}{10mm}\begin{tikzpicture}[scale=0.4]
  \draw (0,0) rectangle (2,2);
  \foreach \x in {0.5,1.5}
    {\fill (\x,2) circle (2pt);
     \fill (\x,0) circle (2pt);}
    \begin{scope} 
    \draw (0.5,2) -- (0.5,0);
        \draw (1.5,2) -- (1.5,0);
   \end{scope}
\end{tikzpicture}\end{minipage} 
  \begin{minipage}{10mm}\begin{tikzpicture}[scale=0.4]
  \draw (0,0) rectangle (2,2);
  \foreach \x in {0.5,1.5}
    {\fill (\x,2) circle (2pt);
     \fill (\x,0) circle (2pt);}
    \begin{scope} 
    \draw (1.5,2) -- (0.5,0);
        \draw (1.5,0) -- (0.5,2);
   \end{scope}
\end{tikzpicture}\end{minipage}  
   \begin{minipage}{10mm}\begin{tikzpicture}[scale=0.4]
  \draw (0,0) rectangle (2,2);
  \foreach \x in {0.5,1.5}
    {\fill (\x,2) circle (2pt);
     \fill (\x,0) circle (2pt);}
    \begin{scope} 
    \draw (0.5,2) -- (0.5,0);
       \draw (0.5,2) arc (180:360:0.5 and 0.3);
              \draw (0.5,0) arc (180:360:0.5 and -0.3);
   \end{scope}
\end{tikzpicture}\end{minipage} 
 \begin{minipage}{10mm}\begin{tikzpicture}[scale=0.4]
  \draw (0,0) rectangle (2,2);
  \foreach \x in {0.5,1.5}
    {\fill (\x,2) circle (2pt);
     \fill (\x,0) circle (2pt);}
    \begin{scope} 
    \draw (0.5,2) -- (0.5,0);
               \draw (0.5,0) arc (180:360:0.5 and -0.3);
   \end{scope}
\end{tikzpicture}\end{minipage} 
  \begin{minipage}{10mm}\begin{tikzpicture}[scale=0.4]
  \draw (0,0) rectangle (2,2);
  \foreach \x in {0.5,1.5}
    {\fill (\x,2) circle (2pt);
     \fill (\x,0) circle (2pt);}
    \begin{scope} 
    \draw (1.5,2) -- (0.5,0);
               \draw (0.5,0) arc (180:360:0.5 and -0.3);
   \end{scope}
\end{tikzpicture}\end{minipage} 
   \begin{minipage}{10mm}\begin{tikzpicture}[scale=0.4]
  \draw (0,0) rectangle (2,2);
  \foreach \x in {0.5,1.5}
    {\fill (\x,2) circle (2pt);
     \fill (\x,0) circle (2pt);}
    \begin{scope}
    \draw (0.5,2) -- (1.5,0);
       \draw (0.5,2) arc (180:360:0.5 and 0.3);
   \end{scope}
\end{tikzpicture}\end{minipage} 
     \begin{minipage}{10mm}\begin{tikzpicture}[scale=0.4]
  \draw (0,0) rectangle (2,2);
  \foreach \x in {0.5,1.5}
    {\fill (\x,2) circle (2pt);
     \fill (\x,0) circle (2pt);}
    \begin{scope}
    \draw (0.5,2) -- (1.5,0);
   \end{scope}
\end{tikzpicture}\end{minipage} 
  \begin{minipage}{10mm}\begin{tikzpicture}[scale=0.4]
  \draw (0,0) rectangle (2,2);
  \foreach \x in {0.5,1.5}
    {\fill (\x,2) circle (2pt);
     \fill (\x,0) circle (2pt);}
    \begin{scope}
    \draw (1.5,2) -- (1.5,0);
   \end{scope}
\end{tikzpicture}\end{minipage} 
\\  \begin{minipage}{10mm}\begin{tikzpicture}[scale=0.4]
  \draw (0,0) rectangle (2,2);
  \foreach \x in {0.5,1.5}
    {\fill (\x,2) circle (2pt);
     \fill (\x,0) circle (2pt);}
    \begin{scope}
    \draw (0.5,2) -- (0.5,0);
   \end{scope}
\end{tikzpicture}\end{minipage}
 \begin{minipage}{10mm}\begin{tikzpicture}[scale=0.4]
  \draw (0,0) rectangle (2,2);
  \foreach \x in {0.5,1.5}
    {\fill (\x,2) circle (2pt);
     \fill (\x,0) circle (2pt);}
    \begin{scope}
    \draw (0.5,2) -- (0.5,0);
       \draw (0.5,2) arc (180:360:0.5 and 0.3);
   \end{scope}
\end{tikzpicture}\end{minipage} 
 \begin{minipage}{10mm}\begin{tikzpicture}[scale=0.4]
  \draw (0,0) rectangle (2,2);
  \foreach \x in {0.5,1.5}
    {\fill (\x,2) circle (2pt);
     \fill (\x,0) circle (2pt);}
    \begin{scope}
    \draw (1.5,2) -- (0.5,0);
     \end{scope}
\end{tikzpicture}\end{minipage}   
  \begin{minipage}{10mm}\begin{tikzpicture}[scale=0.4]
  \draw (0,0) rectangle (2,2);
  \foreach \x in {0.5,1.5}
    {\fill (\x,2) circle (2pt);
     \fill (\x,0) circle (2pt);}
    \begin{scope}
         \draw (0.5,2) arc (180:360:0.5 and 0.3);
              \draw (0.5,0) arc (180:360:0.5 and -0.3);
   \end{scope}
\end{tikzpicture}\end{minipage}  
   \begin{minipage}{10mm}\begin{tikzpicture}[scale=0.4]
  \draw (0,0) rectangle (2,2);
  \foreach \x in {0.5,1.5}
    {\fill (\x,2) circle (2pt);
     \fill (\x,0) circle (2pt);}
    \begin{scope}
               \draw (0.5,0) arc (180:360:0.5 and -0.3);
   \end{scope}
\end{tikzpicture}\end{minipage}   
  \begin{minipage}{10mm}\begin{tikzpicture}[scale=0.4]
  \draw (0,0) rectangle (2,2);
  \foreach \x in {0.5,1.5}
    {\fill (\x,2) circle (2pt);
     \fill (\x,0) circle (2pt);}
    \begin{scope} 
         \draw (0.5,2) arc (180:360:0.5 and 0.3);
   \end{scope}
\end{tikzpicture}\end{minipage}
  \begin{minipage}{10mm}\begin{tikzpicture}[scale=0.4]
  \draw (0,0) rectangle (2,2);
  \foreach \x in {0.5,1.5}
    {\fill (\x,2) circle (2pt);
     \fill (\x,0) circle (2pt);}
    \begin{scope}
   \end{scope}
\end{tikzpicture}\end{minipage} 
  \end{align*}
and multiplication defined by concatenation.  For example:
\begin{align*}
  \begin{minipage}{10mm}
  \begin{tikzpicture}[scale=0.4]
  \draw (0,0) rectangle (2,2);
  \foreach \x in {0.5,1.5}
    {\fill (\x,2) circle (2pt);
     \fill (\x,0) circle (2pt);}
    \begin{scope} 
    \draw (1.5,2) -- (0.5,0);
        \draw (1.5,0) -- (0.5,2);
    \end{scope}
    \draw (0,0) rectangle (2,-2);
  \foreach \x in {0.5,1.5}
    {\fill (\x,-2) circle (2pt);
     \fill (\x,0) circle (2pt);}
    \begin{scope} 
    \draw (0.5,-2) -- (0.5,0);
               \draw (0.5,-2) arc (180:360:0.5 and -0.3);
   \end{scope}
\end{tikzpicture}\end{minipage} =
 \begin{minipage}{10mm}\begin{tikzpicture}[scale=0.4]
  \draw (0,0) rectangle (2,2);
  \foreach \x in {0.5,1.5}
    {\fill (\x,2) circle (2pt);
     \fill (\x,0) circle (2pt);}
    \begin{scope} 
    \draw (1.5,2) -- (0.5,0);
              \draw (0.5,0) arc (180:360:0.5 and -0.3);
   \end{scope}
\end{tikzpicture}\end{minipage} ,
\quad\quad 
 \begin{minipage}{10mm}\begin{tikzpicture}[scale=0.4]
  \draw (0,0) rectangle (2,2);
  \foreach \x in {0.5,1.5}
    {\fill (\x,2) circle (2pt);
     \fill (\x,0) circle (2pt);}
    \begin{scope} 
    \draw (0.5,2) -- (0.5,0);
   \end{scope}
   \draw (0,0) rectangle (2,-2);
  \foreach \x in {0.5,1.5}
    {\fill (\x,-2) circle (2pt);
     \fill (\x,0) circle (2pt);}
    \begin{scope} 
    \draw (0.5,0) -- (1.5,-2);
       \draw (0.5,2) arc (180:360:0.5 and 0.3);
   \end{scope}
\end{tikzpicture}\end{minipage} = n\ \!
 \ \begin{minipage}{10mm}\begin{tikzpicture}[scale=0.4]
  \draw (0,0) rectangle (2,2);
  \foreach \x in {0.5,1.5}
    {\fill (\x,2) circle (2pt);
     \fill (\x,0) circle (2pt);}
    \begin{scope} 
    \draw (0.5,2) -- (1.5,0);
 \draw (0.5,2) arc (180:360:0.5 and 0.3);
   \end{scope}
\end{tikzpicture}\end{minipage} 
\end{align*}

There are four standard modules corresponding to the partitions of degree less than or equal to $2$; these are obtained by inflating the Specht modules from the symmetric groups of degree $0,1,2$.  These modules have bases:

$$\begin{array}{llll}
  &\Delta_2(2)= {\rm Span}_{\mathbb{C}} \left\{ \, \begin{minipage}{8.5mm}\begin{tikzpicture}[scale=0.4]
  \draw (0,0) rectangle (2,2);
  \foreach \x in {0.5,1.5}
    {\fill (\x,2) circle (2pt);
     \fill (\x,0) circle (2pt);}
    \begin{scope} 
    \draw (0.5,2) -- (0.5,0);
        \draw (1.5,2) -- (1.5,0);
   \end{scope}
\end{tikzpicture}\end{minipage} +
  \begin{minipage}{8.5mm}\begin{tikzpicture}[scale=0.4]
  \draw (0,0) rectangle (2,2);
  \foreach \x in {0.5,1.5}
    {\fill (\x,2) circle (2pt);
     \fill (\x,0) circle (2pt);}
    \begin{scope} 
    \draw (1.5,2) -- (0.5,0);
        \draw (1.5,0) -- (0.5,2);
   \end{scope} 
\end{tikzpicture}\end{minipage}\, \right\}  
  &\Delta_2(1^2) = {\rm Span}_{\mathbb{C}} 
\left\{  \, \begin{minipage}{8.5mm}\begin{tikzpicture}[scale=0.4]
  \draw (0,0) rectangle (2,2);
  \foreach \x in {0.5,1.5}
    {\fill (\x,2) circle (2pt);
     \fill (\x,0) circle (2pt);}
    \begin{scope} 
    \draw (0.5,2) -- (0.5,0);
        \draw (1.5,2) -- (1.5,0);
   \end{scope}
\end{tikzpicture}\end{minipage} -
  \begin{minipage}{8.5mm}\begin{tikzpicture}[scale=0.4]
  \draw (0,0) rectangle (2,2);
  \foreach \x in {0.5,1.5}
    {\fill (\x,2) circle (2pt);
     \fill (\x,0) circle (2pt);}
    \begin{scope} 
    \draw (1.5,2) -- (0.5,0);
        \draw (1.5,0) -- (0.5,2);
   \end{scope}
\end{tikzpicture}\end{minipage} \,\right\} \\
\end{array}
$$

$$
\begin{array}{llll}
  &\Delta_2(1) =  {\rm Span}_{\mathbb{C}} \left\{\, \begin{minipage}{8.5mm}\begin{tikzpicture}[scale=0.4]
  \draw (0,0) rectangle (2,2);
  \foreach \x in {0.5,1.5}
    {\fill (\x,2) circle (2pt);
     \fill (\x,0) circle (2pt);}
    \begin{scope} 
    \draw (0.5,2) -- (0.5,0);
       \draw (0.5,2) arc (180:360:0.5 and 0.3);
    \end{scope}
\end{tikzpicture}\end{minipage} ,
 \begin{minipage}{8.5mm}\begin{tikzpicture}[scale=0.4]
  \draw (0,0) rectangle (2,2);
  \foreach \x in {0.5,1.5}
    {\fill (\x,2) circle (2pt);
     \fill (\x,0) circle (2pt);}
    \begin{scope}
    \draw (1.5,2) -- (0.5,0);
    \end{scope}
\end{tikzpicture}\end{minipage} ,
 \begin{minipage}{8.5mm}\begin{tikzpicture}[scale=0.4]
  \draw (0,0) rectangle (2,2);
  \foreach \x in {0.5,1.5}
    {\fill (\x,2) circle (2pt);
     \fill (\x,0) circle (2pt);}
    \begin{scope} 
    \draw (0.5,2) -- (0.5,0);
   \end{scope}
\end{tikzpicture}\end{minipage}\, \right\}
     &\Delta_2(\emptyset) =   {\rm Span}_{\mathbb{C}}\left\{\,\begin{minipage}{8.5mm}\begin{tikzpicture}[scale=0.4]
  \draw (0,0) rectangle (2,2);
  \foreach \x in {0.5,1.5}
    {\fill (\x,2) circle (2pt);
     \fill (\x,0) circle (2pt);}
    \begin{scope} 
         \draw (0.5,2) arc (180:360:0.5 and 0.3);
          \draw (0.5,0) arc (180:360:0.5 and -0.3);
    \end{scope}
\end{tikzpicture}\end{minipage} ,
  \begin{minipage}{8.5mm}\begin{tikzpicture}[scale=0.4]
  \draw (0,0) rectangle (2,2);
  \foreach \x in {0.5,1.5}
    {\fill (\x,2) circle (2pt);
     \fill (\x,0) circle (2pt);}
    \begin{scope} 
          \draw (0.5,0) arc (180:360:0.5 and -0.3);
   \end{scope}
\end{tikzpicture}\end{minipage}\,\right\}
  \end{array}$$

The action of $P_2(n)$ is given by concatenation.  If the resulting diagram has fewer propagating lines than the original, we set the product equal to zero.
The algebra $P_1(n)\otimes P_1(n)$ is the 4-dimensional subalgebra spanned by the diagrams with no lines crossing an imagined vertical wall down the centre of the diagram. 
The restriction of the standard modules to this subalgebra  is  as follows:
$$\Delta_2(2)\!\!\downarrow_{P_1\otimes P_1}\cong \Delta_1(1) \boxtimes  \Delta_1(1),\quad  \Delta_2(1^2)\!\!\downarrow_{P_1\otimes P_1}\cong \Delta_1(1) \boxtimes  \Delta_1(1),$$
$$\Delta_2(1)\!\!\downarrow_{P_1\otimes P_1}\cong \Delta_1(1) \boxtimes  \Delta_1(1) \oplus \Delta_1(\emptyset) \boxtimes  \Delta_1(1) \oplus \Delta_1(1) \boxtimes  \Delta_1(\emptyset),$$
$$\Delta_2(\emptyset)\!\!\downarrow_{P_1\otimes P_1}\cong \Delta_1(1) \boxtimes  \Delta_1(1) \oplus \Delta_1(\emptyset) \boxtimes  \Delta_1(\emptyset) .$$
In particular, note that $\bar{g}_{(1),(1)}^\nu = [\Delta_2(\nu)\!\!\downarrow_{P_1\otimes P_1}:\Delta_1(1)\boxtimes\Delta_1(1)]=1$ for $\nu = \emptyset,1,1^2,2$.   

The partition algebra $P_2(n)$ is semisimple for $n>2$.  For $\nu=\emptyset, (1), (1^2)$ or $(2)$ we have that $\nu_{[n]}=(n), (n-1,1), (n-2,1^2),$ or $(n-2,2)$ and $\nu_{[n]}$ is  a partition for $n\geq 0,2,3,4$ respectively.  Therefore the  Kronecker coefficients $$g^{\nu_{[n]}}_{(n-1,1),(n-1,1)} $$ stabilise for $n\geq 4$ and are non-zero for $n\geq 4$ if and only $\nu_{[n]}$ is one of the partitions above.

Now consider the case $n=2$.  Neither $\nu=(1^2),$ nor $(2)$ correspond to partitions of 2, we therefore consider $\nu=\emptyset$ and $(1)$.  We have that $(1) \subset (2)$  is the unique $2$-pair of partitions of degree less than or equal to 2 (see Section 3.3). Therefore the only standard $P_2(2)$-module which is not simple is $\Delta_2(1)$ and we have an exact sequence
$$0\to L_2(2) \to \Delta_2(1) \to L_2(1) \to 0.$$
Thus in the Grothendieck group we have that  $[L_2(1)]=[\Delta_2(1)]-[\Delta_2(2)]$.  Hence, we have that $[L_2(1)\!\!\downarrow_{P_1(2)\otimes P_1(2)}:L_1(1) \boxtimes L_1(1)]=0$.  We conclude that $g_{(1^2), (1^2)}^{(1^2)}=0$ and $g_{(1^2),(1^2)}^{(2)}=1$ as expected.

\section*{Acknowledgements}
The authors wish to thank David Speyer for pointing out a crucial error in an earlier version of Theorem \ref{partition} and Corollary \ref{bigthm}.
M. De Visscher and R. Orellana thank Georgia Benkart, Monica Vazirani and Stephanie van Willigenburg and the Banff International Research Station for providing support and a stimulating environment during the Algebraic Combinatorixx workshop where this project started. 
C. Bowman and R. Orellana are grateful for the financial support received from the ANR and NSF grants ANR-10-BLAN-0110 and DMS-1101740, respectively. 

\bibliographystyle{amsalpha} \bibliography{book3}

\providecommand{\bysame}{\leavevmode\hbox to3em{\hrulefill}\thinspace}
\providecommand{\MR}{\relax\ifhmode\unskip\space\fi MR }
\providecommand{\MRhref}[2]{%
  \href{http://www.ams.org/mathscinet-getitem?mr=#1}{#2}
}
\providecommand{\href}[2]{#2}
\begin{thebibliography}{BOR11}

\bibitem[BK99]{BK}
C.~Bessenrodt and A.~Kleshchev, \emph{On {K}ronecker products of complex
  representations of the symmetric and alternating groups}, Pacific Journal of
  Mathematics \textbf{190} (1999), no.~201-223.

\bibitem[Bla12]{blasiak}
J.~Blasiak, \emph{Kronecker coefficients of one hook shape}, arxiv:1209.2209v2,
  2012.

\bibitem[BO05]{BallantineOrellanaEJC}
C.~Ballantine and R.~Orellana, \emph{On the kronecker product of $s_{n-p,p}\ast
  s_{\lambda}$}, Electron. J. Combin. \textbf{12} (2005), no.~\# R28, 1--26.

\bibitem[BOR11]{rosa}
E.~Briand, R.~Orellana, and M.~Rosas, \emph{The stability of the {K}ronecker
  product of {S}chur functions}, J. Algebra \textbf{331} (2011), 11--27.

\bibitem[Bri93]{brion}
M.~Brion, \emph{Stable properties of plethysm: on two conjectures of
  {F}oulkes}, Manuscripta Math. \textbf{80} (1993), 347--371.

\bibitem[DR92]{dr1}
V.~Dlab and C.~M. Ringel, \emph{The module theoretic approach to
  quasi-hereditary algebras}, Representations of algebras and related topics
  (H.~Tachikawa and S.~Brenner, eds.), LMS Lecture Note Series, vol. 168, 1992,
  pp.~200--224.

\bibitem[Dvi93]{Dvir}
Y.~Dvir, \emph{On the {K}ronecker product of ${S}_n$ characters}, J. Algebra
  \textbf{154} (1993), no.~1, 125--140.

\bibitem[GR85]{GR}
A.~M. Garsia and J.~Remmel, \emph{Shuffles of permutations and the {K}ronecker
  product}, Graphs Combin. \textbf{1} (1985), no.~3, 217--263.

\bibitem[GW98]{goodwall}
R.~Goodman and N.~R. Wallach, \emph{Representations and invariants of the
  classical groups}, CUP, 1998.

\bibitem[HR05]{HalvRam}
T.~Halverson and A.~Ram, \emph{{Partition algebras}}, {European J. Combin.}
  \textbf{{26}} ({2005}), no.~{6}, {869--921}.

\bibitem[JK81]{jk}
G.~D. James and A.~Kerber, \emph{The representation theory of the symmetric
  group}, Encyclopedia of Mathematics and its Applications, vol.~16,
  Addison-Wesley, 1981.

\bibitem[Jon94]{Jones}
V.~F.~R. Jones, \emph{The {P}otts model and the symmetric group}, In:
  Subfactors: Proceedings of the Tanaguchi Symposium on Operator Algebras
  (Kyuzeso, 1993) (NJ), World Sci. Publishing River Edge, 1994, pp.~259--267.

\bibitem[Kly04]{klyachko}
A.~Klyachko, \emph{Quantum marginal problem and representations of the
  symmetric group}, arXiv:quant-ph/0409113 (2004).

\bibitem[Las80]{Lascoux}
A.~Lascoux, \emph{Produit de {K}ronecker des repr\'esentations du groupe
  sym\'etrique}, Lecture Notes in Math., pp.~319--329, Berlin, 1980.

\bibitem[Lit58]{littlewood}
D.~E. Littlewood, \emph{Products and plethysms of characters with orthogonal,
  symplectic and symmetric groups}, Cand. J. Math \textbf{10} (1958), 17--32.

\bibitem[Mac95]{mac}
I.~G. Macdonald, \emph{Symmetric functions and {H}all polynomials}, 2nd ed.,
  Oxford Mathematical Monographs., The Claredon Press Oxford University Press,
  New York,, 1995.

\bibitem[Mar91]{marbook}
P.~P. Martin, \emph{Potts models and related problems in statistical
  mechanics}, Series on Advances in Statistical Mechanics, 5, World Scientific
  Publishing Co., Inc., Teaneck, NJ, 1991.

\bibitem[Mar96]{mar1}
\bysame, \emph{The structure of the partition algebras}, J. Algebra
  \textbf{183} (1996), 319--358.

\bibitem[Mur38]{murn}
F.~Murnaghan, \emph{The analysis of the {K}ronecker product of irreducible
  representations of the symmetric group}, Amer. J. Math. \textbf{60} (1938),
  no.~3, 761--784.

\bibitem[Mur55]{murn2}
\bysame, \emph{On the analysis of the {K}ronecker product of irreducible
  representations of ${S}_n$}, Proc. Nat. Acad. Sci. U.S.A. \textbf{41} (1955),
  515--518.

\bibitem[Thi91]{Thibon}
J.~Thibon, \emph{Hopf algebras of symmetric functions and tensor products of
  symmetric group representations}, Int. J. Algebra Comput. \textbf{1} (1991),
  no.~2, 207--221.

\bibitem[Val99]{vallejo}
E.~Vallejo, \emph{Stability of {K}ronecker products of irreducible characters
  of the symmetric group}, Electron. J. Combin. \textbf{6} (1999), no.~1, 1--7.

\end{thebibliography}

\end{document}